\documentclass[11pt]{article}
\usepackage{a4wide}
\usepackage{amssymb,amsmath}
\usepackage{color}
\usepackage[oldvoltagedirection]{circuitikz}
\newtheorem{theorem}{Theorem}[section]

\newtheorem{lemma}[theorem]{Lemma}
\newtheorem{corollary}[theorem]{Corollary}

\newtheorem{definition}[theorem]{Definition}
\newtheorem{example}[theorem]{Example}
\newtheorem{remark}[theorem]{Remark}

\newcommand {\diag}     {\mathop{\rm diag}\nolimits}
\newcommand {\rank}     {\mathop{\rm rank}\nolimits}

\newcommand {\kernel}   {\mathop{\rm kernel}\nolimits}

\newcommand {\pspan}    {\mathop{\rm span}\nolimits}

\newcommand {\simnew}{\stackrel{\hbox to 0pt{\rm\scriptsize\hss new\hss}}{\sim}}
\newcommand {\simdif}{\stackrel{\hbox to 0pt{\rm\scriptsize\hss dif\hss}}{\sim}}
\newcommand {\ddt}{\vphantom{\tilde E}{\textstyle{d\over dt}}}

\def\.{^{\vphantom{(}}}

\def\mat#1{\left[\begin{array}{#1}}
\def\rix{\end{array}\right]}

\def\L#1{{\mathbb L}^{#1}\kern-4pt\raise-2pt\hbox{$\scriptstyle{\mu^{#1}}$}}
\def\inner#1{{\mathop{\mathbb #1}\limits^{\kern3pt\raise-3pt\hbox{$\scriptscriptstyle\circ$}}}{}}

\newcommand{\adots}{\mbox{\setlength{\unitlength}{1pt}
                          \begin{picture}(8,7)\put(-4,1){.}\put(0,4){.}
                          \put(4,7){.}\end{picture}}}
\title{Local and global canonical forms for differential-algebraic equations with symmetries\footnotemark[1]
}
\author{Peter Kunkel\footnotemark[2]    \and
 Volker Mehrmann\footnotemark[3]
 }
\begin{document}

\maketitle
{\bf Abstract}. Linear time-varying differential-algebraic  equations with symmetries are studied.
The structures that we address are self-adjoint and skew-adjoint systems. Local and global canonical forms under congruence are presented and used to classify the geometric properties of
the flow associated with the differential equation as symplectic or generalized orthogonal flow.
As applications, the results are applied to the analysis of dissipative Hamiltonian systems arising from circuit simulation and incompressible flow.
\vskip .3truecm

\vskip .3truecm

{\bf AMS(MOS) subject classification.}   37J05, 65L80, 65L05, 65P10,  49K15

\parindent=10pt

\footnotetext[1]{Partially supported by the Research In Pairs program of
{\sl Mathematisches Forschungsinstitut Oberwolfach},
whose hospitality is gratefully acknowledged.}

\footnotetext[2]{
Mathematisches Institut, Universit\"at Leipzig, Augustusplatz 10,
D-04109 Leipzig, Fed.\ Rep.\ Germany, \texttt{kunkel@math.uni-leipzig.de}.
}

\footnotetext[3]{
Institut f\"ur Mathematik, TU Berlin, Str.\ des 17.~Juni 136,
D-10623 Berlin, Fed.\ Rep.\ Germany, \texttt{mehrmann@math.tu-berlin.de}.
Partially supported by the {\it Deutsche Forschungsgemeinschaft} through
Project A2 of CRC 910 \emph{Control of self-organizing nonlinear systems: Theoretical methods and concepts of application}}

\vspace{0.5cm}
\centerline{\bf Dedicated to Alfi Quarteroni on the occasion of his 70th birthday.}

\section{Introduction}\label{sec:intro}

We study regular linear variable coefficient differential-algebraic equations (DAEs)
\begin{equation}\label{lindae}
E(t) \dot x =A(t) x+f(t),\quad
\begin{array}{l}
E,A\in C({\mathbb I},{\mathbb R}^{n,n}),\
f\in C({\mathbb I},{\mathbb R}^{n})\mbox{ sufficiently smooth},
\end{array}
\end{equation}
which additionally possess certain symmetries, in particular self-adjoint and skew-adjoint structures.
Here ${\mathbb I}\subseteq{\mathbb R}$ is a compact non-trivial time interval
and $C^k({\mathbb I},{\mathbb R}^{n,n})$ with $k\in{\mathbb N}_0\cup\{\infty\}$ denotes the set of $k$ times continuously differentiable functions
from $\mathbb I$ into the set of real $n\times n$ matrices ${\mathbb R}^{n,n}$.
In the case $k=0$ we drop the superscript. A function is said to be sufficiently smooth
if $k$ is sufficiently large such that all needed derivatives exist.

The discussed classes of DAEs with symmetries are defined as follows.
\begin{definition}\label{def:self}
The DAE (\ref{lindae}) and its associated pair $(E,A)$ of matrix functions are called
{\em self-adjoint} if
\begin{equation}\label{self}
E^T=-E,\quad A^T=A+\dot E
\end{equation}
as equality of functions.
\end{definition}

\begin{definition}\label{def:skew}
The DAE (\ref{lindae}) and its associated pair $(E,A)$ of matrix functions are called
{\em skew-adjoint} if
\begin{equation}\label{skew}
E^T=E,\quad A^T=-A-\dot E
\end{equation}
as equality of functions.
\end{definition}

Systems with these type of symmetries arise in the modeling of physical systems,
e.g., by a (generalized) Hamiltonian formalism or in optimal control problems
leading to a self-adjoint structure or by so-called dissipative Hamiltonian systems
leading to a skew-adjoint structure.

Our main motivation to study differential-algebraic equations with self-adjoint and skew-adjoint pairs arises
from  multi-physics, multi-scale models that are coupling different physical domains that may include mechanical, mechatronic, fluidic, thermic, hydraulic, pneumatic, elastic, plastic, or electric components, see e.g.  \cite{AltMU21,AltS17,EicF98,HilH06,Sch93,SchK01}.

An important class of problems where multi-physics and multi-scale modeling arises is  the human cariovascular system, where model hierarchies of detailed models for the blood flow in large vessels, modeled via the Navier-Stokes equations, and reduced or surrogate models for the capillary vessels, modeled via electrical network equations, are coupled together to improve computational efficiency, while at the same time achieving a desired simulation accuracy, see \cite{ForGNQ01,ForQV10,QuaF04,QuaMV17,QuaLRR17}. Due to the physical background, after space discretization and linearization  along a solution, as well as ignoring dissipation terms, all the components arising in this application can be expressed as  DAE systems with symmetries.

To deal with general multi-physics and multi-scale coupling in recent years  the framework of  port-Hamiltonian  systems has become an important modeling paradigm  \cite{BeaMXZ18,GolSBM03,JacZ12,MehM19,OrtSMM01,Sch04,Sch06} that encodes underlying physical principles  directly into the algebraic structure of the coefficient matrices and in the geometric structure associated with the flow of the dynamical system. This leads to a remarkably flexible modeling approach, which has also been extended to include algebraic constraints, so that the resulting model is  a port-Hamiltonian differential-algebraic equation (pHDAE), \cite{BeaMXZ18,MehM19,Sch13}. Such systems allow for automated modeling in a modularized network based fashion, and they are ideal for building model hierarchies of very fine models for numerical simulation and reduced or surrogate models for control and optimization. This makes them particularly suited also for large networks, such as power, gas, or district heating networks where such model hierarchies are used to adapt the simulation and optimization techniques to user needs, \cite{DomHLMMT21,HauMMMMRS19,MehM19}.

Since in this paper we will mainly discuss linear time-varying DAE systems with symmetries, we introduce the structure of pHDAEs for this case as in \cite{BeaMXZ18}, see \cite{MehM19} for the more general nonlinear framework.

Linear time-varying \emph{pHDAE systems with quadratic Hamiltonian}  have the form
\begin{subequations}
		\label{eqn:pHDAE:linear}
		\begin{align}
			E(t)\dot{x} + E(t)K(t)x &=(J(t)-R(t))x + (G(t)-P(t))u,\\
			y &= (G(t)+P(t))^T x + (S(t)-N(t))u,
		\end{align}
\end{subequations}
with state $x$, input $u$, output $y$ and coefficients
$E\in C^1(\mathbb I,\mathbb R^{n,n})$, $J,R,K \in C(\mathbb I,\mathbb R^{n,n})$, $G,P\in C(\mathbb I,\mathbb R^{n, m})$,
$S,N\in C(\mathbb I,\mathbb R^{m, m})$, $S = S^T, N=-N^T$.
As energy function one has the quadratic \emph{Hamiltonian}
	\begin{equation}
	 		\label{eqn:Hamiltonian:linear}
	 		\mathcal H \colon \mathbb I \times \mathbb R^{n}\to \mathbb R,\qquad (t,x) \mapsto \tfrac{1}{2}x^H E(t)x,
	 \end{equation}
and  the pair of coefficients $(E, J -  EK)$ is skew-adjoint, while the matrix function associated with dissipation of energy
$$
\mathcal W = \begin{bmatrix}
	 				R & P\\
	 				P^T & S
	 			\end{bmatrix}
$$
is pointwise positive semidefinite. Furthermore, typically one also has that $E$ is pointwise positive semidefinite as well. If the system does not have an output equation and the input is considered as an inhomogeneity then this is called a dissipative Hamiltonian DAE (dHDAE).

The underlying skew-adjoint structure  arises if the dissipation term $R$ is neglected, i.e., if $\mathcal W=0$ in \eqref{eqn:pHDAE:linear}. Typically the problems  with dissipation can be considered as a perturbed symmetry  structure and a dissipative term can be treated separately as a by-product in simulation methods.

To illustrate applications for dHDAEs and DAEs with symmetries consider the following simple examples, for further applications see \cite{BeaMXZ18,MehM19,Sch13,SchJ14}.

\begin{example} \label{ex1} {\rm Consider the pHDAE formulation  of an electrical circuit from \cite{MehM19}.
\begin{figure}[ht]
  \centering
  \begin{circuitikz}[scale=1.5,/tikz/circuitikz/bipoles/length=1cm]
    \draw (0,0) node[ground] {} to[american controlled voltage source,invert,v>=$E_G$] (0,1) to[R=$R_G$,i>=$I_G$] (0,2) to (1,2) to[L=$L$,i>=$I$] (2,2) to[R=$R_L$] (3,2) to (4,2) to[R=$R_R$,i<=$I_R$] (4,0) node[ground] {};
    \draw (1,2) to[C=$C_1$,v>=$V_1$,i<=$I_1$,*-] (1,0) node[ground] {};
    \draw (3,2) to[C,l_=$C_2$,v^>=$V_2$,i<_=$I_2$,*-] (3,0) node[ground] {};
  \end{circuitikz}
  \caption{circuit example}\label{fig:circuit}
\end{figure}
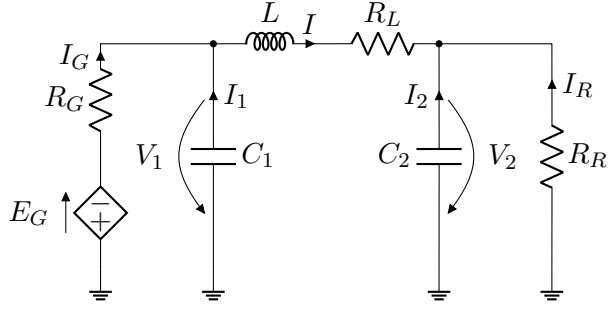
Denoting by $V_i$ the voltages and by $I_i$ the currents, where $L>0$ models an inductor, $C_1,C_2>0$  capacitors, $R_G,R_L,R_R>0$ resistances, and $E_G$ a controlled voltage source, one obtains a pHDAE
\begin{subequations}\label{eq:circuit2}
  \begin{align}
    E\dot x &= (J-R) x + G u, \\
y &= G^T x,
  \end{align}
\end{subequations}
with
$x=\begin{bmatrix} I &V_1&V_2&I_G&I_R\end{bmatrix}^T $, $u=E_G$, $y=I_G$, $E=\diag(L,C_1,C_2,0,0)$,
\begin{equation*}
G=\begin{bmatrix} 0 \\ 0 \\ 0 \\ 1 \\ 0\end{bmatrix},\quad
  J = \begin{bmatrix} 0 & -1 & 1 & 0 & 0 \\ 1 & 0 & 0 & -1 & 0 \\ -1 & 0 & 0 & 0 & -1 \\ 0 & 1 & 0 & 0 & 0 \\ 0 & 0 & 1 & 0 & 0\end{bmatrix}, \quad
  R = \begin{bmatrix}R_L & 0 & 0 & 0 & 0 \\ 0 & 0 & 0 & 0 & 0 \\ 0 & 0 & 0 & 0 & 0 \\ 0 & 0 & 0 & R_G & 0 \\ 0 & 0 & 0 & 0 & R_R\end{bmatrix}.
\end{equation*}
  The quadratic Hamiltonian, describing the energy stored in the inductor and the two capacitors, is given by
\[
  \mathcal H(I,V_1,V_2) = \frac12LI^2 + \frac12C_1V_1^2 + \frac12C_2V_2^2.
\]
If the generator is shut down (i.e.\ $E_G=0$), then the system approaches an equilibrium solution for which $\ddt\mathcal H(x)=0$, so that $I=I_G=I_R=0$ and thus $x=0$ in the equilibrium.
Without the damping by the resistances (i.e.\ setting $R_L=R_G=R_R=0$) this is a skew-adjoint DAE.
}
\end{example}
\begin{example}\label{ex:ns}{\rm
A classical example of a partial differential equation which, after proper space discretization leads to a pHDAE,
see e.g.\ \cite{EmmM13,QuaQ09}, are the  incompressible or nearly incompressible
{\em Navier-Stokes equations} describing the flow of a Newtonian fluid in a domain $\Omega$,
\begin{align*}
\partial_t{v} &= \nu \Delta v - (v\cdot\nabla) v - \nabla p +
f \qquad \text{in } \Omega \times \mathbb T ,\\
0 &=\nabla^T v  \qquad \text{in } \Omega \times \mathbb T,
\end{align*}
together with suitable initial and boundary conditions, see e.g.\ \cite{Tem77}. When one linearizes around a prescribed  vector field $v_{\infty}$, then one
obtains the  \emph{linearized Navier-Stokes equations},
%
%
and if $v_{\infty}$ is constant in space
and the term $(v_{\infty}\cdot\nabla) v$ is neglected then one obtains the \emph{Stokes equation}.

Performing a finite element discretization in space, see for instance \cite{Lay08}, a Galerkin projection leads to a dHDAE of the form
\begin{equation}
    \label{eq:P-instat-op-general}
    \begin{bmatrix} M & 0 \\ 0 & 0 \end{bmatrix}
    \begin{bmatrix} \dot{v} \\ \dot{p} \end{bmatrix} =
    \begin{bmatrix} A_S(t)-A_H(t) & -B\\ B^T & -C\end{bmatrix}
    \begin{bmatrix} v\\p\end{bmatrix}+\begin{bmatrix}f(t)\\ 0\end{bmatrix},
\end{equation}
where $M=M^T>0$ is the mass matrix, $A_S=-A_S^T$, $A_H=A_H^T\geq 0$ are skew-symmetric and symmetric part of the discretized and linearized convection-diffusion operator, $B$ is the discretized gradient operator, $B^T$ is the discretized divergence operator, which we assume to be normalized so that it is of full row rank, and $C=C^T>0$ is a stabilization term of small norm that is needed for some
finite element spaces, see e.g.\ \cite{BreF91,QuaQ09,Ran00,RooST08}. Here $v$ and $p$ denote the discretized velocity and pressure, respectively, and $f$ is a forcing term.
Without the damping (i.e.\ for $A_H=0$ and $C=0$) we again have a skew-adjoint DAE.
}
\end{example}

The class of problems with self-adjoint structure  arises  most prominently in the context of constrained generalized Hamiltonian systems and in  optimal control problems, where the operators associated with the optimality conditions have this structure.

\begin{example}\label{ex:OCP} {\rm In \cite{KunM08} the optimality conditions were derived for
the linear-quadratic optimal control problem of minimizing a cost functional
\begin{equation}\label{linOCP}
{\mathcal J}(x,u)=\frac 12  x(t_f)^T  M_f x(t_f)+\frac 12
\int_{t_0}^{t_f}\left(x^T  W x +x^T  S u +u^T S^T x +u^T  R u\right)\,dt ,
\end{equation}
subject to the DAE constraint
\begin{equation}\label{DAE_1}
E \dot x= A x+ B u + f,\quad x(t_0)={x_0}\in{\mathbb R}^{n},
\end{equation}
with
$E,A,W\in  C({\mathbb I},{\mathbb R}^{n,n})$,
$B \in C({\mathbb I},{\mathbb R}^{n,m})$,
$S \in C({\mathbb I},{\mathbb R}^{n,m})$,
$R \in C({\mathbb I},{\mathbb R}^{m,m})$,
$f \in C({\mathbb I},{\mathbb R}^{n})$ and $M_e \in {\mathbb R}^{n,n}$, where
$R=R^T$, $W=W^T$ and $M_f=M_f^T$.

After some appropriate reformulation (via some index reduction process) and under  some smoothness conditions, the optimality condition is given by a boundary value problem
\begin{equation*}\label{e:DAE_op}
\mat{ccc}
                           0 & E & 0\\
                           -E^T & 0 & 0\\
                           0 & 0 & 0\rix \frac{d}{dt}
                           \mat{c} \lambda \\ x \\ u\rix=
\mat{ccc}
          0 & A & B\\
          A^T+\frac{d}{dt} E^T & W & S\\
          B^T & S^T & R\rix\mat{c} \lambda \\ x \\ u\rix+
          \mat{c} f\\ 0 \\ 0\rix,
\end{equation*}
with boundary conditions $x(t_0)={x_0}$,
$E(t_f)^T \lambda(t_f)- M_f x(t_f)=0$.
The associated pair of coefficient functions obviously is a a self-adjoint pair, see \cite{KunMS14}.
}
\end{example}

\begin{example}\label{ex:conHam}{\rm
Linear multibody systems with linear holonomic constraints, see \cite{HaiLW02,LeiR94}, take the form
{\arraycolsep 2.4pt
\begin{eqnarray*}
M \dot p & =& -W q -G^T \lambda,\\
\dot q & = & p,\\
0 & =&G q,
\end{eqnarray*}}%
where $p,q \in C^1(\mathbb I, \mathbb R^{n})$,
$W,M\in \mathbb R^{n,n}$ with $W=W^T$, $M=M^T$,
and $G\in \mathbb R^{m,n}$.
If the mass matrix is positive definite, i.e.\ $M>0$,
then we can multiply the second equation by $-M$ and the constraint by $-1$ to obtain, after switching the first and second equation,
\[
\mat{ccc} 0 & M & 0 \\ -M & 0 & 0 \\ 0 & 0 & 0 \rix \mat{c} \dot q \\ \dot p \\ \dot \lambda \rix=
\mat{ccc} -W & 0 & -G^T \\ 0 & -M & 0 \\ -G & 0 & 0 \rix \mat{c} q \\ p \\ \lambda \rix,
\]
which has the structure of a self-adjoint DAE.

If $W>0$ then we can also multiply the second equation of the constrained Hamiltonian system
by~$W$ to obtain
\[
\mat{ccc} W & 0 & 0 \\ 0 & M & 0 \\  0 & 0 & 0 \rix \mat{c} \dot q \\ \dot p \\ \dot \lambda \rix=
\mat{ccc} 0 & W & 0 \\ -W & 0 &  -G^T \\ 0 & G & 0 \rix \mat{c} q \\ p \\ \lambda \rix
\]
which now has the structure of a skew-adjoint DAE.
}
\end{example}

\begin{remark} \label{rem:trafo}{\rm
It should be noted that if in a self-adjoint system $E \dot x=A x+f$, both  $E$ and $A$ are constant in time and invertible, then by multiplication with $E^{-1}$ and a change of variables $z= Ax$ we get a system
 $ A^{-1} \dot z = E^{-1} z+ E^{-1} f$ that is skew-adjoint.
 A similar construction has also been proposed for dissipative Hamiltonian systems in \cite{Egg19}.
 }
\end{remark}

Having illustrated the importance of DAEs with symmetries, in this paper we present canonical forms for DAEs (\ref{lindae})
with self-adjoint and skew-adjoint structure and show the consequences for the resulting flows.

The paper is organized as follows. In Section~\ref{sec:prelim} we recall some results for general DAEs. In Section~\ref{sec:self} we discuss canonical forms for self-adjoint and in Section~\ref{sec:skew} for skew-adjoint DAEs.

\section{Preliminaries}\label{sec:prelim}
Linear time-varying DAE systems have been extensively discussed, see \cite{KunM06} for a detailed analytical and numerical treatment.  In this section we recall some basic concepts from the general theory of DAEs.
Our first concept is that of \emph{regularity} for DAEs that is concerned with the existence of solutions at least for
sufficiently smooth inhomogeneity (surjectivity) and uniqueness
of the solution in the cases where the initial condition $x(t_0)=x_0$
allows for a solution (injectivity).

\begin{definition}\label{def:regular}
The pair $(E,A)$ and the corresponding DAE\ (\ref{lindae})
are called \emph{regular} if
\begin{enumerate}
\item
the DAE\ (\ref{lindae}) is solvable for every sufficiently smooth~$f$,
\item
the solution is unique
for every $t_0\in{\mathbb I}$ and every
$x_0\in{\mathbb R}^n$ allowing for a solution of the DAE with $x(t_0)=x_0$,
\item
the solution depends smoothly on~$f$, $t_0$, and~$x_0$.
\end{enumerate}
\end{definition}

The most important technique for the analysis of general linear DAEs is the construction of suitable local
and global canonical forms under global equivalence. Since we will refer to these results and some techniques from
their derivation, we include here the necessary material from the general (square) case.

We start with (global) equivalence which refers to time-dependent scaling
of the DAE and changes of basis.
\begin{definition}\label{def:global}
Two pairs $(E_i,A_i)$, $E_i,A_i\in C({\mathbb I},{\mathbb R}^{n,n})$, $i=1,2$,
of matrix functions
are called {\em $($globally$)$ equivalent} if there exist pointwise nonsingular
matrix functions
$P\in C({\mathbb I},{\mathbb R}^{n,n})$ and
$Q\in C^1({\mathbb I},{\mathbb R}^{n,n})$ such that
\begin{equation}\label{global}
E_2=PE_1Q,\quad A_2=PA_1Q-PE_1\dot Q
\end{equation}
as equality of functions.
We write $(E_1,A_1)\sim(E_2,A_2)$.
\end{definition}
It is easy to see that the relation defined in Definition~\ref{def:global}
indeed is an equivalence relation, \cite{KunM06}.

The derivation of canonical forms then relies on the following property
of matrix functions on intervals, see~\cite{Dol64,KunM06}.

\begin{theorem}\label{th:rank}
Let $E\in C^k({\mathbb I},{\mathbb R}^{m,n})$,
$k\in{\mathbb N}_0\cup\{\infty\}$,
with $\rank E(t)=r$ for all $t\in{\mathbb I}$.
Then there exist pointwise orthogonal functions
$U\in C^k({\mathbb I},{\mathbb R}^{m,m})$ and
$V\in C^k({\mathbb I},{\mathbb R}^{n,n})$ such that
\begin{equation}\label{sep}
U^TEV= \begin{bmatrix} \Sigma&0\\0&0\end{bmatrix}
\end{equation}
with pointwise nonsingular $\Sigma\in C^k({\mathbb I},{\mathbb R}^{r,r})$.
\end{theorem}

We then have the following result on a local canonical form, i.~e.,
a canonical form that requires the restriction to certain subintervals, see~\cite[Section~3.1,Section~3.3]{KunM06}.
In the following, non-specified blocks of matrices or matrix functions are denoted by~$*$.

\begin{theorem}\label{th:lcfgen}
Let $(E,A)$ be regular with $E,A\in C({\mathbb I},{\mathbb R}^{n,n})$ sufficiently smooth.
Then there exist pairwise disjoint open intervals ${\mathbb I}_j$, $j\in{\mathbb N}$, with
\begin{equation}\label{dec}
\overline{\bigcup_{j\in{\mathbb N}}{\mathbb I}_j}={\mathbb I}
\end{equation}
such that on every ${\mathbb I}_j$ one has
\begin{equation}\label{lcfgen}
(E,A)\sim\left(\mat{cc}I_d&W\\0&G\rix,\mat{cc}0&0\\0&I_a\rix\right),
\end{equation}
where $G$ is structurally nilpotent according to
\begin{equation}\label{nil}
G=\mat{cccc}0&*&\cdots&*\\&\ddots&\ddots&\vdots\\&&\ddots&*\\&&&0\rix.
\end{equation}
Furthermore, the size~$d$ of the differential part and the size~$a$ of the algebraic part
are the same for every interval.
\end{theorem}

A global canonical form, i.e., a canonical form that does not require
the restriction to certain subintervals, was given in~\cite{Cam87a}.
We state this result here in a version for real-valued problems omitting
the last step of the proof which would require complex-valued transformations.
We include the proof for later reference.

\begin{theorem}\label{th:gcf}
Let $(E,A)$ be regular with $E,A\in C({\mathbb I},{\mathbb R}^{n,n})$ sufficiently smooth.
Then we have
\begin{equation}\label{gcf}
(E,A)\sim\left(
\left[\begin{array}{cc} I_{d}&E_{12}\\0&E_{22} \end{array}\right],
\left[\begin{array}{cc} 0&A_{12}\\0&A_{22} \end{array}\right]\right),
\end{equation}
where
\begin{equation}\label{gcfq}
E_{22}(t)\dot x_2=A_{22}(t)x_2+f_2(t)
\end{equation}
is uniquely solvable for every sufficiently smooth~$f_2$
without specifying initial conditions.
\end{theorem}
\begin{proof}
If the homogeneous equation
\[
E(t)\dot x=A(t)x
\]
has only the trivial solution, then
the first block in (\ref{gcf}) is missing (i.~e., $d=0$) and the
claim holds trivially by assumption. In any case, the solution space is
finite dimensional, since otherwise we could not select a unique solution
by prescribing initial conditions.

Let therefore $d\ne0$ and let $\{\phi_1,\ldots,\phi_{d}\}$
be a basis of the solution space. Setting
$\Phi=[\>\phi_1\>\>\cdots\>\>\phi_{d}\>]$,
we have
\[
\hbox{$\rank \Phi(t)=d$ for all $t\in{\mathbb I}$,}
\]
since, if we had $\rank \Phi(t)<d$ for some $t_0\in{\mathbb I}$,
then there would exist coefficients
$\alpha_1,\ldots,\alpha_{d}\in{\mathbb R}$,
not all being zero, with
\[
\alpha_1\phi_1(t_0)+\cdots+\alpha_{d}\phi_{d}(t_0)=0
\]
and $\alpha_1\phi_1+\cdots+\alpha_{d}\phi_{d}$
would be a nontrivial solution of the homogeneous initial value problem.

Hence, by Theorem~\ref{th:rank}
there exists a smooth, pointwise nonsingular matrix function~$U$ with
\[
U^H\Phi=\left[\begin{array}{c} I_{d}\\0 \end{array}\right].
\]
Defining
\[
\Phi'=U\left[\begin{array}{c} 0\\I_{a} \end{array}\right]
\]
yields a pointwise nonsingular matrix function $Q=[\>\Phi\>\>\Phi'\>]$.
Since $E\dot\Phi=A\Phi$, we obtain
\[
(E,A)\sim
([\>E\Phi\>\>E\Phi'\>],
[\>A\Phi\>\>A\Phi'\>]-[\>E\dot\Phi\>\>E\dot\Phi'\>])=
([\>E_1\>\>E_2\>],[\>0\>\>A_2\>]).
\]
In this relation, $E_1$ has full column rank~$d$.
To see this, suppose that $\rank E_1(\hat t)<d$
for some $\hat t\in{\mathbb I}$.
Then there would exist a vector $w\ne0$ with
\[
E_1(\hat t)w=0.
\]
Defining in this situation
\[
f(t)=\left\{\begin{array}{ll}
{1\over t-\hat t}E_1(t)w&\hbox{for $t\ne\hat t$},\\[1mm]
{d\over dt}(E_1(t)w)&\hbox{for $t=\hat t$} ,
\end{array}\right.
\]
we would obtain a smooth inhomogeneity~$f$. The function~$x$ given by
\[
x(t)=\left[\begin{array}{c} \log(\vert t-\hat t\vert)w\\0 \end{array}\right]
\]
would then solve
\[
[\>E_1(t)\>\>E_2(t)\>]\dot x=[\>0\>\>A_2(t)\>]x+f(t)
\]
on ${\mathbb I}\setminus\{\hat t\}$ in contradiction to the assumption
of unique solvability, which includes by definition that solutions are defined
on the entire interval~${\mathbb I}$.

Hence, since $E_1$ has full column rank,
there exists a smooth, pointwise nonsingular matrix function~$P$ with
\[
PE_1=\left[\begin{array}{c} I_{d}\\0 \end{array}\right],
\]
and thus
\[
(E,A)\sim\left(
\left[\begin{array}{cc} I_{d}&E_{12}\\0&E_{22} \end{array}\right],
\left[\begin{array}{cc} 0&A_{12}\\0&A_{22} \end{array}\right]\right).
\]
The equation
\[
E_{22}(t)\dot x_2=A_{22}(t)x_2
\]
only admits the trivial solution. To see this, suppose that
$x_2\ne0$ is a nontrivial solution
and~$x_1$ a solution of the ODE
\[
\dot x_1+E_{12}(t)\dot x_2(t)=A_{22}(t)x_2(t).
\]
Then we obtain
\[
[\>E_1(t)\>\>E_2(t)\>]
\left[\begin{array}{c} \dot x_1(t)\\\dot x_2(t) \end{array}\right]=
[\>0\>\>A_2(t)\>]
\left[\begin{array}{c} x_1(t)\\x_2(t) \end{array}\right].
\]
Since $x_2\ne0$, transforming back gives
\[
E(t)Q(t)
\left[\begin{array}{c} \dot x_1(t)\\\dot x_2(t) \end{array}\right]=
(A(t)Q(t)-E(t)\dot Q(t))
\left[\begin{array}{c} x_1(t)\\x_2(t) \end{array}\right]
\]
or $E(t)\dot x(t)=A(t)x(t)$ with
\[
x=Q\left[\begin{array}{c} x_1\\x_2\end{array}\right]\ne0,\quad
x\not\in\pspan\{\phi_1,\ldots,\phi_{d}\}.
\]
But this contradicts the construction of $\phi_1,\ldots,\phi_{d}$,
and hence (\ref{gcf}) holds.
\end{proof}

In the presence of symmetries, we of course want to maintain these properties,
which  requires to restrict the allowed equivalence transformations.
We will make use of the following notions and properties.
\begin{definition}\label{def:congruence}
Two pairs $(E_i,A_i)$, $E_i,A_i\in C({\mathbb I},{\mathbb R}^{n,n})$, $i=1,2$,
of matrix functions
are called {\em congruent} if there exist a pointwise nonsingular
matrix function
$Q\in C^1({\mathbb I},{\mathbb R}^{n,n})$ such that
\begin{equation}\label{congruence}
E_2=Q^TE_1Q,\quad A_2=Q^TA_1Q-Q^TE_1\dot Q
\end{equation}
as equality of functions.
We write $(E_1,A_1)\equiv(E_2,A_2)$.
\end{definition}

Again, it is easy to see that the relation defined in Definition~\ref{def:congruence}
indeed is an equivalence relation, see \cite{KunMS14}.

The following result then modifies Theorem~\ref{th:rank} provided some symmetry property holds.
\begin{theorem}\label{th:ranksym}
Let $E\in C^k({\mathbb I},{\mathbb R}^{n,n})$,
$k\in{\mathbb N}_0\cup\{\infty\}$,
with $\rank E(t)=r$ for all $t\in{\mathbb I}$
and let $\kernel E(t)^T=\kernel E(t)$ for all $t\in{\mathbb I}$.
Then there exist a pointwise orthogonal function
$Q\in C^k({\mathbb I},{\mathbb R}^{n,n})$ such that
\begin{equation}\label{sepsym}
Q^TEQ= \begin{bmatrix} \Sigma&0\\0&0\end{bmatrix}
\end{equation}
with pointwise nonsingular $\Sigma\in C^k({\mathbb I},{\mathbb R}^{r,r})$.
\end{theorem}
\begin{proof}
With $Q=V$ from Theorem~\ref{th:rank}, the pointwise property $\kernel E^T=\kernel E$
allows to choose $U=Q^T$.
\end{proof}

In the next two sections we will employ  these preliminary results to derive
canonical forms for self-adjoint and skew-adjoint DAEs.

\section{Canonical forms for self-adjoint pairs of matrix functions}\label{sec:self}

In this section we study canonical forms under congruence for self-adjoint DAEs.
We first show that congruence preserves the self-adjoint structure.
\begin{lemma}\label{lem:self} Consider two pairs of matrix functions
 $(E,A)$ and $(\widetilde E,\widetilde A)$ that are congruent and let $(E,A)$ be self-adjoint.
Then $(\widetilde E,\widetilde A)$ is self-adjoint as well.
\end{lemma}
\begin{proof}
Let
\[
\widetilde E=Q^TEQ,\quad\widetilde A=Q^TAQ-Q^TE\dot Q
\]
and assume that (\ref{self}) holds. Then
\[
\widetilde E^T=Q^TE^TQ=-Q^TEQ=-\widetilde E
\]
and
\[
\begin{array}{l}
\widetilde A^T=
Q^TA^TQ-\dot Q^TE^TQ=Q^T(A+\dot E)Q+\dot Q^TEQ\\[1mm]
\hphantom{\widetilde A^T}=
Q^TAQ-Q^TE\dot Q+Q^TE\dot Q+Q^T\dot EQ+\dot Q^TEQ\\
\hphantom{\widetilde A^T}=
(Q^TAQ-Q^TE\dot Q)+\ddt(Q^TEQ)=\widetilde A+\dot{\widetilde E}.
\end{array}
\]
\end{proof}

For self-adjoint pairs the following local canonical form under pointwise orthogonal
congruence transformations is due to \cite{KunMS14} stated here for the special
case of a regular pair of matrix functions.

\begin{theorem}\label{th:lcfself1}
Let $(E,A)$ be regular with $E,A\in C({\mathbb I},{\mathbb R}^{n,n})$ sufficiently smooth and let $(E,A)$ be skew-adjoint. Then
there exist pairwise disjoint open intervals~${\mathbb I}_j$, $j\in{\mathbb N}$, with (\ref{sep})
such that on every~${\mathbb I}_j$ there exists a pointwise orthogonal~$Q\in C({\mathbb I},{\mathbb R}^{n,n})$ with
\begin{equation}\label{lcfself1}
(Q^TEQ,
Q^T\!AQ-Q^TE\dot Q)=
\left(\mat{c|cc|c}
*&*&*&E_{14}\\\hline
*&\Delta&0&0\\
*&0&0&0\\\hline
E_{41}&0&0&0
\rix,\mat{c|cc|c}
*&*&*&A_{14}\\\hline
*&\Sigma_{11}&\Sigma_{12}&0\\
*&\Sigma_{21}&\Sigma_{22}&0\\\hline
A_{41}&0&0&0
\rix\right),
\end{equation}
where
\begin{equation}\label{nilself1}
E_{14}=\mat{cccc}
*&\cdots&*&0\\
\vdots&\adots&\adots\\
*&\adots\\
0
\rix,\quad
A_{14}=\mat{cccc}
*&\cdots&*&\Gamma_w\\
\vdots&\adots&\adots\\
*&\adots\\
\Gamma_1
\rix
\end{equation}
and $\Delta,\Sigma_{22},\Gamma_1,\ldots,\Gamma_w$ are pointwise nonsingular. Furthermore,
\begin{equation}\label{propself1}
\Delta^T=-\Delta,\quad\Sigma_{11}^T=\Sigma_{11}\.+\dot\Delta,\quad
\Sigma_{21}^T=\Sigma_{12}\.,\quad\Sigma_{22}^T=\Sigma_{22}\.,\quad A_{41}^T=A_{14}\.+\dot E_{14}.
\end{equation}
\end{theorem}

Theorem~\ref{th:lcfself1} can be further refined  by allowing for a restricted class of non-orthogonal transformations, see again~\cite{KunMS14}.

\begin{theorem}\label{th:lcfself2}
Let $(E,A)$ be regular with $E,A\in C({\mathbb I},{\mathbb R}^{n,n})$ sufficiently smooth and let $(E,A)$ be skew-adjoint. Then
there exist pairwise disjoint open intervals~${\mathbb I}_j$, $j\in{\mathbb N}$, with (\ref{sep})
such that on every~${\mathbb I}_j$ there exists a pointwise nonsingular~$Q\in C({\mathbb I},{\mathbb R}^{n,n})$ with
\begin{equation}\label{lcfself2}
(Q^TEQ,
Q^T\!AQ-Q^TE\dot Q)=
\left(\mat{c|cc|c}
*&*&*&E_{14}\\\hline
*&J&0&0\\
*&0&0&0\\\hline
E_{41}&0&0&0
\rix,\mat{c|cc|c}
*&*&*&A_{14}\\\hline
*&C&0&0\\
*&0&\Sigma_{22}&0\\\hline
A_{41}&0&0&0
\rix\right),
\end{equation}
where
\begin{equation}\label{nilself2}
E_{14}=\mat{ccccc}
*&\cdots&*&0\\
\vdots&\adots&\adots\\
*&\adots\\
0
\rix,\quad
A_{14}=\mat{ccccc}
*&\cdots&*&I\\
\vdots&\adots&\adots\\
*&\adots\\
I
\rix.
\end{equation}
and $\Sigma_{22}$ pointwise nonsingular. Furthermore,
\begin{equation}\label{propself2}
J=\mat{cc}0&I_p\\-I_p&0\rix,\quad C^T=C\.,\quad
\Sigma_{22}^T=\Sigma_{22}\.,\quad A_{41}^T=A_{14}\.+\dot E_{14}\..
\end{equation}
\end{theorem}

By successively resolving the algebraic equations in the fourth, third and first row of the DAE with solution
$[\>x_1^T \>\> x_2^T \>\> x_3^T \>\> x_4^T \>]^T$ and inhomogeneity
$[\>f_1^T \>\> f_2^T \>\> f_3^T \>\> f_4^T \>]^T$
associated with \eqref{lcfself2}, we can directly solve for $x_1$ in terms of linear combinations of derivatives of $f_4$,
for $x_3$ in terms of $x_1$ and $f_3$, and for $x_4$ in terms of linear combinations of derivatives of $f_1$ and all other components. The only dynamic behavior related to (\ref{lcfself2}) is described
by the second block row. Inserting $x_1$ obtained from the last block row and calling the updated inhomogeneity $\widetilde f_2$, the associated ODE reads
\begin{equation}\label{odeself}
\dot x_2=J^{-1}C(t)x_2+J^{-1}\widetilde f_2(t).
\end{equation}
The matrix function $M=J^{-1} C$ satisfies $M^TJ-JM=0$ and lies therefore pointwise in the \emph{Lie algebra of Hamiltonian matrices}, see e.g.\ \cite{Hal03}. Thus, the flow corresponding to (\ref{odeself})
defined by
\begin{equation}\label{odeselffunda}
\dot \Phi_2=J^{-1}C(t) \Phi_2,\quad \Phi_2(t_0)= I
\end{equation}
satisfies $\Phi_2^T J \Phi_2 =J$, see e.g.\ \cite{HaiLW02}, and is therefore symplectic.

In~\cite{KunMS14}, also a global canonical form for self-adjoint DAEs and the associated pairs $(E,A)$ was derived.
The following modified result differs slightly in the assumptions, and,  moreover, the resulting canonical form is more refined.

\begin{theorem}\label{th:gcfself}
Let $(E,A)$ be regular with $E,A\in C({\mathbb I},{\mathbb R}^{n,n})$ sufficiently smooth and let $(E,A)$ be self-adjoint. Then we have
\begin{equation}\label{gcfself}
(E,A)\equiv\left(
\left[\begin{array}{ccc} 0&I_{p}&0\\-I_{p}&0&0\\0&0&E_{33} \end{array}\right],
\left[\begin{array}{ccc} 0&0&0\\0&A_{22}&A_{23}\\0&A_{32}&A_{33} \end{array}\right]\right),
\end{equation}
where
\begin{equation}\label{gcfselfq}
E_{33}(t)\dot x_3=A_{33}(t)x_3+f_3(t),
\end{equation}
is uniquely solvable for every sufficiently smooth~$f_3$
without specifying initial conditions. Furthermore,
\begin{equation}\label{gcfselfp}
E_{33}^T=-E_{33}\.,\quad A_{22}^T=A_{22}\.,\quad A_{32}^T=A_{23}\.,\quad A_{33}^T=A_{33}+\dot E_{33}.
\end{equation}
\end{theorem}
\begin{proof}
According to Theorem~\ref{th:lcfself2}, the size~$d$ of the differential part
is given by $d=2p$. This implies that the solution space of the
homogeneous equation
\[
E(t)\dot x=A(t)x
\]
is of dimension~$2p$. If $p=0$, then
the first two blocks are missing and the
claim holds trivially by assumption.

Let therefore $p\ne0$ and let $\{\phi_1,\ldots,\phi_{2p}\}$
be a basis of the solution space. Setting
$\Phi=[\>\phi_1\>\>\cdots\>\>\phi_{2p}\>]$,
we have
\[
\hbox{$\rank \Phi(t)=2p$ for all $t\in{\mathbb I}$}
\]
as in the general case of Theorem~\ref{th:gcf}.

Hence, by Theorem~\ref{th:rank}
there exists a smooth, pointwise nonsingular matrix function~$U$ with
\[
U^H\Phi=\left[\begin{array}{c} I_{2p}\\0 \end{array}\right].
\]
Defining
\[
\Phi'=U\left[\begin{array}{c} 0\\I_{a} \end{array}\right]
\]
with $a=n-2p$
yields a pointwise nonsingular matrix function $Q=[\>\Phi\>\>\Phi'\>]$.
Since $E\dot\Phi=A\Phi$, we obtain
\[
(\widetilde E,\widetilde A)=(Q^TEQ,Q^TAQ-Q^TE\dot Q)
\]
with
\[
\begin{array}{l}
\widetilde E=\mat{cc}\Phi^TE\Phi&\Phi^TE\Phi'\\\Phi'^TE\Phi&\Phi'^TE\Phi'\rix=
\mat{cc}\widetilde E_{11}&\widetilde E_{12}\\-\widetilde E_{12}^T&\widetilde E_{22}\rix,\\
\widetilde A=\mat{cc}\Phi^T(A\Phi-E\dot\Phi)&\Phi^T(A\Phi'-E\dot\Phi')\\\Phi'^T(A\Phi-E\dot\Phi)&\Phi'^T(A\Phi'-E\dot\Phi')\rix=
\mat{cc}0&\widetilde A_{12}\\0&\widetilde A_{22}\rix.
\end{array}
\]
To simplify the notation, we omit now and later at similar instances
the tildes thus re-using the same notation for possibly different quantities
and write
\[
(E,A)\equiv\left(
\mat{cc}E_{11}&E_{12}\\-E_{12}^T&E_{22}\rix,
\mat{cc}0&A_{12}\\0&A_{22}\rix\right).
\]
As in the general case, we can conclude that
\[
\rank\mat{c}E_{11}\\-E_{12}^T\rix=2p.
\]
Since self-adjointness is conserved, we additionally have
\[
E_{11}^T=-E_{11},\quad E_{22}^T=-E_{22},\quad 0=\dot E_{11},\quad
0=A_{12}+\dot E_{12},\quad A_{22}^T=A_{22}+\dot E_{22}.
\]
In particular, $E_{11}$ is constant and skew-symmetric.
Hence, see \cite{KunMS14},  there exists an orthogonal symplectic matrix~$U\in{\mathbb R}^{2p,2p}$ with
\[
\begin{array}{l}
\widetilde E=
\mat{cc}U^T&0\\0&I_a\rix
\mat{cc}E_{11}&E_{12}\\-E_{12}^T&E_{22}\rix
\mat{cc}U&0\\0&I_a\rix=
\mat{ccc}0&\widetilde E_{12}&\widetilde E_{13}\\-\widetilde E_{12}^T&\widetilde E_{22}&\widetilde E_{23}\\-\widetilde E_{13}^T&-\widetilde E_{23}^T&\widetilde E_{33}\rix,\\
\widetilde A=
\mat{cc}U^T&0\\0&I_a\rix
\mat{cc}0&A_{12}\\0&A_{22}\rix
\mat{cc}U&0\\0&I_a\rix=
\mat{ccc}0&0&\widetilde A_{13}\\0&0&\widetilde A_{23}\\0&0&\widetilde A_{33}\rix,
\end{array}
\]
where $\widetilde E_{12}\in{\mathbb R}^{p,p}$.

Omitting again the tildes, we write
\[
(E,A)\equiv\left(
\mat{ccc}0&E_{12}&E_{13}\\-E_{12}^T&E_{22}&E_{23}\\-E_{13}^T&-E_{23}^T&E_{33}\rix,
\mat{ccc}0&0&A_{13}\\0&0&A_{23}\\0&0&A_{33}\rix\right).
\]
Conservation of self-adjointness and full rank of the leading block yields
\[
E_{22}^T=-E_{22},\quad E_{33}^T=-E_{33},\quad 0=\dot E_{12},\quad 0=\dot E_{22},\quad\rank[\>E_{12}\>\>E_{13}\>]=p.
\]
Hence, there exists a smooth, pointwise nonsingular matrix function~$V$ with
\[
[\>E_{12}\>\>E_{13}\>]V=[\>I_p\>\>0\>]
\]
leading to
\[
\begin{array}{l}
\widetilde E=
\mat{cc}I_p&0\\0&V^T\rix
\mat{ccc}0&E_{12}&E_{13}\\-E_{12}^T&E_{22}&E_{23}\\-E_{13}^T&-E_{23}^T&E_{33}\rix
\mat{cc}I_p&0\\0&V\rix=
\mat{ccc}0&I_p&0\\-I_p&\widetilde E_{22}&\widetilde E_{23}\\0&-\widetilde E_{23}^T&\widetilde E_{33}\rix,\\
\widetilde A=
\mat{cc}I_p&0\\0&V^T\rix
\mat{ccc}0&0&\widetilde A_{13}\\0&0&\widetilde A_{23}\\0&0&\widetilde A_{33}\rix
\mat{cc}I_p&0\\0&V\rix\\
\qquad\qquad{}-
\mat{cc}I_p&0\\0&V^T\rix
\mat{ccc}0&E_{12}&E_{13}\\-E_{12}^T&E_{22}&E_{23}\\-E_{13}^T&-E_{23}^T&E_{33}\rix
\mat{cc}0&0\\0&\dot V\rix=
\mat{ccc}0&\widetilde A_{12}&\widetilde A_{13}\\0&\widetilde A_{22}&\widetilde A_{23}\\0&\widetilde A_{32}&\widetilde A_{33}\rix.
\end{array}
\]
Omitting the tildes again, we write
\[
(E,A)\equiv\left(
\mat{ccc}0&I_p&0\\-I_p&E_{22}&E_{23}\\0&-E_{23}^T&E_{33}\rix,
\mat{ccc}0&A_{12}&A_{13}\\0&A_{22}&A_{23}\\0&A_{32}&A_{33}\rix\right).
\]
Conservation of self-adjointness yields
\[
E_{22}^T=-E_{22},\quad E_{33}^T=-E_{33},\quad 0=A_{12},\quad 0=A_{13}.
\]
Finally, after a congruence transformation with,
\[
\mat{ccc}I_p&\frac12E_{22}&E_{23}\\0&I_p&0\\0&0&I_a\rix,
\]
we arrive at
\[
\arraycolsep 4pt
\begin{array}{l}
\widetilde E=
\mat{ccc}I_p&0&0\\-\frac12E_{22}&I_p&0\\E_{23}^T&0&I_a\rix
\mat{ccc}0&I_p&0\\-I_p&E_{22}&E_{23}\\0&-E_{23}^T&E_{33}\rix
\mat{ccc}I_p&\frac12E_{22}&E_{23}\\0&I_p&0\\0&0&I_a\rix=
\mat{ccc}0&I_p&0\\-I_p&0&0\\0&0&\widetilde E_{33}\rix,\\
\widetilde A=
\mat{ccc}I_p&0&0\\-\frac12E_{22}&I_p&0\\E_{23}^T&0&I_a\rix
\mat{ccc}0&A_{12}&A_{13}\\0&A_{22}&A_{23}\\0&A_{32}&A_{33}\rix
\mat{ccc}I_p&\frac12E_{22}&E_{23}\\0&I_p&0\\0&0&I_a\rix\\
\qquad\qquad{}-
\mat{ccc}I_p&0&0\\-\frac12E_{22}&I_p&0\\E_{23}^T&0&I_a\rix
\mat{ccc}0&I_p&0\\-I_p&E_{22}&E_{23}\\0&-E_{23}^T&E_{33}\rix
\mat{ccc}0&\frac12\dot E_{22}&\dot E_{23}\\0&0&0\\0&0&0\rix=
\mat{ccc}0&0&0\\0&\widetilde A_{22}&\widetilde A_{23}\\0&\widetilde A_{32}&\widetilde A_{33}\rix.
\end{array}
\]
Omitting again the tildes then yields
\[
(E,A)\equiv\left(
\mat{ccc}0&I_p&0\\-I_p&0&0\\0&0&E_{33}\rix,
\mat{ccc}0&0&0\\0&A_{22}&A_{23}\\0&A_{32}&A_{33}\rix\right)
\]
which is just (\ref{gcfself}),
where (\ref{gcfselfq}) follows along the same lines as in the general case
and (\ref{gcfselfp}) follows by the conservation of self-adjointness.
\end{proof}

After having derived canonical forms for self-adjoint DAEs, in the next section we derive an analogous form for skew-adjoint DAEs.

\section{Canonical forms for skew-adjoint pairs of matrix functions}\label{sec:skew}

In this section we show that canonical forms under congruence can be derived also for skew-adjoint DAEs.
\begin{lemma}\label{lem:skew} Consider two pairs of matrix functions
 $(E,A)$ and $(\widetilde E,\widetilde A)$ that are congruent
 and let $(E,A)$ be skew-adjoint.
Then $(\widetilde E,\widetilde A)$ is skew-adjoint as well.
\end{lemma}
\begin{proof}
Let
\[
\widetilde E=Q^TEQ,\quad\widetilde A=Q^TAQ-Q^TE\dot Q
\]
and assume that (\ref{skew}) holds. Then
\[
\widetilde E^T=Q^TE^TQ=Q^TEQ=\widetilde E
\]
and
\[
\begin{array}{l}
\widetilde A^T=
Q^TA^TQ-\dot Q^TE^TQ=Q^T(-A-\dot E)Q-\dot Q^TEQ\\[1mm]
\hphantom{\widetilde A^T}=
-Q^TAQ+Q^TE\dot Q-Q^TE\dot Q-Q^T\dot EQ-\dot Q^TEQ\\
\hphantom{\widetilde A^T}=
-(Q^TAQ-Q^TE\dot Q)-\ddt(Q^TEQ)=-\widetilde A-\dot{\widetilde E}.
\end{array}
\]
\end{proof}

The following result on a local canonical form under pointwise orthogonal
congruence transformations is due to \cite{Sch19}.
\begin{theorem}\label{th:lcfskew1}
Let $(E,A)$ be regular with $E,A\in C({\mathbb I},{\mathbb R}^{n,n})$ sufficiently smooth and let $(E,A)$ be skew-adjoint. Then
there exist pairwise disjoint open intervals~${\mathbb I}_j$, $j\in{\mathbb N}$, with (\ref{sep})
such that on every~${\mathbb I}_j$ there exists a pointwise orthogonal~$Q\in C({\mathbb I},{\mathbb R}^{n,n})$ with
\begin{equation}\label{lcfskew1}
(Q^TEQ,
Q^T\!AQ-Q^TE\dot Q)=
\left(\mat{c|cc|c}
*&*&*&E_{14}\\\hline
*&\Delta&0&0\\
*&0&0&0\\\hline
E_{41}&0&0&0
\rix,\mat{c|cc|c}
*&*&*&A_{14}\\\hline
*&\Sigma_{11}&\Sigma_{12}&0\\
*&\Sigma_{21}&\Sigma_{22}&0\\\hline
A_{41}&0&0&0
\rix\right),
\end{equation}
where
\begin{equation}\label{nilskew1}
E_{14}=\mat{cccc}
*&\cdots&*&0\\
\vdots&\adots&\adots\\
*&\adots\\
0
\rix,\quad
A_{14}=\mat{cccc}
*&\cdots&*&\Gamma_w\\
\vdots&\adots&\adots\\
*&\adots\\
\Gamma_1
\rix
\end{equation}
and $\Delta,\Sigma_{22},\Gamma_1,\ldots,\Gamma_w$ are pointwise nonsingular. Furthermore,
\begin{equation}\label{propskew1}
\Delta^T=\Delta,\quad\Sigma_{11}^T=-\Sigma_{11}\.-\dot\Delta,\quad
\Sigma_{21}^T=-\Sigma_{12}\.,\quad\Sigma_{22}^T=-\Sigma_{22}\.,\quad A_{41}^T=-A_{14}\.-\dot E_{14}.
\end{equation}
\end{theorem}

Theorem~\ref{th:lcfself1} can be further refined by allowing for a restricted class of non-orthogonal transformations, which yields the following local canonical form which, using a recent result of~\cite{Kun20}, is more refined than that in~\cite{Sch19}.
\begin{theorem}\label{th:lcfskew2}
Let $(E,A)$ be regular with $E,A\in C({\mathbb I},{\mathbb R}^{n,n})$ sufficiently smooth and let $(E,A)$ be skew-adjoint. Then
there exist pairwise disjoint open intervals~${\mathbb I}_j$, $j\in{\mathbb N}$, with (\ref{sep})
such that on every~${\mathbb I}_j$ there exists a pointwise nonsingular~$Q\in C({\mathbb I},{\mathbb R}^{n,n})$ with
\begin{equation}\label{lcfskew2}
(Q^TEQ,
Q^T\!AQ-Q^TE\dot Q)=
\left(\mat{c|cc|c}
*&*&*&E_{14}\\\hline
*&S&0&0\\
*&0&0&0\\\hline
E_{41}&0&0&0
\rix,\mat{c|cc|c}
*&*&*&A_{14}\\\hline
*&J&0&0\\
*&0&\Sigma_{22}&0\\\hline
A_{41}&0&0&0
\rix\right),
\end{equation}
where
\begin{equation}\label{nilskew2}
E_{14}=\mat{ccccc}
*&\cdots&*&0\\
\vdots&\adots&\adots\\
*&\adots\\
0
\rix,\quad
A_{14}=\mat{ccccc}
*&\cdots&*&I\\
\vdots&\adots&\adots\\
*&\adots\\
I
\rix.
\end{equation}
and $\Sigma_{22}$ pointwise nonsingular. Furthermore,
\begin{equation}\label{propskew2}
S=\mat{cc}I_p&0\\0&-I_q\rix,\quad J^T=-J\.,\quad
\Sigma_{22}^T=-\Sigma_{22}\.,\quad A_{41}^T=-A_{14}\.-\dot E_{14}\..
\end{equation}
\end{theorem}
\begin{proof}
Compared with the result in~\cite{Sch19}, to obtain the stated local canonical form
we need a smooth version of Sylvester's law of inertia, which is proved in~\cite{Kun20}.
In particular, this result shows the
existence of a smooth transformation~$W$ with $W^T\!\Delta\,W=S$, where $\Delta$ is
pointwise nonsingular and symmetric.
\end{proof}

By successively resolving the algebraic equations in the fourth, third and first row of the DAE with solution
$[\>x_1^T \>\> x_2^T \>\> x_3^T \>\> x_4^T \>]^T$ and inhomogeneity
$[\>f_1^T \>\> f_2^T \>\> f_3^T \>\> f_4^T \>]^T$
associated with \eqref{lcfself2}, we can directly solve for $x_1$ in terms of linear combinations of derivatives of $f_4$,
for $x_3$ in terms of $x_1$ and $f_3$, and for $x_4$ in terms of linear combinations of derivatives of $f_1$ and all other components. The only dynamic behavior related to (\ref{lcfself2}) is described
by the second block row. Inserting $x_1$ obtained from the last block row and calling the updated inhomogeneity $\widetilde f_2$, the associated ODE reads
\begin{equation}\label{odeskew}
\dot x_2=S^{-1}J(t)x_2+S^{-1}\widetilde f_2(t).
\end{equation}
The matrix function $M=S^{-1}J$ has the property that $SM$ is skew-symmetric,
i.e., that $(SM)^T=-(SM)$ or $M^TS+SM=0$, and lies therefore pointwise in the Lie algebra
belonging to the quadratic Lie group
\[
O(p,q)=\{\Phi\in{\mathbb R}^{n,n}\mid\Phi^TS\Phi=S\},
\]
the so-called \emph{generalized orthogonal group with inertia $(p,q,0)$}, see e.g.\ \cite{Hal03},
implying that the flow belonging to (\ref{odeskew}) lies in $O(p,q)$.
Observe the special case $p=0$ or $q=0$ where the flow is then orthogonal.

We also obtain a global canonical form.

\begin{theorem}\label{th:gcfskew}
Let $(E,A)$ be regular with $E,A\in C({\mathbb I},{\mathbb R}^{n,n})$ sufficiently smooth and let $(E,A)$ be skew-adjoint. Then we have
\begin{equation}\label{gcfskew}
(E,A)\equiv\left(
\left[\begin{array}{ccc} I_{p}&0&0\\0&-I_{q}&0\\0&0&E_{33} \end{array}\right],
\left[\begin{array}{ccc} 0&0&0\\0&0&0\\0&0&A_{33} \end{array}\right]\right),
\end{equation}
where
\begin{equation}\label{gcfskewq}
E_{33}(t)\dot x_3=A_{33}(t)x_3+f_3(t)
\end{equation}
is uniquely solvable for every sufficiently smooth~$f_3$
without specifying initial conditions. Furthermore,
\begin{equation}\label{gcfskewp}
E_{33}^T=E_{33}\.,\quad A_{33}^T=-A_{33}-\dot E_{33}\.
\end{equation}
\end{theorem}
\begin{proof}
According to Theorem~\ref{th:lcfskew2}, the size~$d$ of the differential part
is given by $d=p+q$. This implies that the solution space of the
homogeneous equation
\[
E(t)\dot x=A(t)x
\]
is of dimension~$p+q$. If $p+q=0$, then
the first two blocks are missing and the
claim holds trivially by assumption.

Let therefore $p+q\ne0$ and let $\{\phi_1,\ldots,\phi_{p+q}\}$
be a basis of the solution space. Setting
$\Phi=[\>\phi_1\>\>\cdots\>\>\phi_{p+q}\>]$,
we have
\[
\hbox{$\rank \Phi(t)=p+q$ for all $t\in{\mathbb I}$}
\]
as in the general case of Theorem~\ref{th:gcf}.

Hence, by Theorem~\ref{th:rank}
there exists a smooth, pointwise nonsingular matrix function~$U$ with
\[
U^H\Phi=\left[\begin{array}{c} I_{p+q}\\0 \end{array}\right].
\]
Defining
\[
\Phi'=U\left[\begin{array}{c} 0\\I_{a} \end{array}\right]
\]
with $a=n-(p+q)$
yields a pointwise nonsingular matrix function $Q=[\>\Phi\>\>\Phi'\>]$.
Since $E\dot\Phi=A\Phi$, we obtain
\[
(\widetilde E,\widetilde A)=(Q^TEQ,Q^TAQ-Q^TE\dot Q)
\]
with
\[
\begin{array}{l}
\widetilde E=\mat{cc}\Phi^TE\Phi&\Phi^TE\Phi'\\\Phi'^TE\Phi&\Phi'^TE\Phi'\rix=
\mat{cc}\widetilde E_{11}&\widetilde E_{12}\\ \widetilde E_{12}^T&\widetilde E_{22}\rix,\\
\widetilde A=\mat{cc}\Phi^T(A\Phi-E\dot\Phi)&\Phi^T(A\Phi'-E\dot\Phi')\\\Phi'^T(A\Phi-E\dot\Phi)&\Phi'^T(A\Phi'-E\dot\Phi')\rix=
\mat{cc}0&\widetilde A_{12}\\0&\widetilde A_{22}\rix.
\end{array}
\]
As in the proof of Theorem~\ref{th:gcfself}, we omit now and later at similar instances
the tildes thus re-using the same notation for possibly different quantities
and write
\[
(E,A)\equiv\left(
\mat{cc}E_{11}&E_{12}\\E_{12}^T&E_{22}\rix,
\mat{cc}0&A_{12}\\0&A_{22}\rix\right).
\]
As in the general case, we can conclude that
\[
\rank\mat{c}E_{11}\\E_{12}^T\rix=p+q.
\]
Since skew-adjointness is conserved, we additionally have
\[
E_{11}^T=E_{11},\quad E_{22}^T=E_{22},\quad 0=\dot E_{11},\quad
0=-A_{12}-\dot E_{12},\quad A_{22}^T=-A_{22}-\dot E_{22}.
\]
In particular, $E_{11}$ is constant and symmetric.
Moreover, in the following we will show that~$E_{11}$ is nonsingular.

Due to Sylvester's law of inertia, there is a nonsingular matrix~$U\in{\mathbb R}^{d,d}$,
$d=p+q+r$, with
\[
\begin{array}{l}
\widetilde E=
\mat{cc}U^T&0\\0&I_a\rix
\mat{cc}E_{11}&E_{12}\\E_{12}^T&E_{22}\rix
\mat{cc}U&0\\0&I_a\rix=
\mat{ccc}S&0&\widetilde E_{13}\\0&0&\widetilde E_{23}\\\widetilde E_{13}^T&\widetilde E_{23}^T&\widetilde E_{33}\rix,\\
\widetilde A=
\mat{cc}U^T&0\\0&I_a\rix
\mat{cc}0&A_{12}\\0&A_{22}\rix
\mat{cc}U&0\\0&I_a\rix=
\mat{ccc}0&0&\widetilde A_{13}\\0&0&\widetilde A_{23}\\0&0&\widetilde A_{33}\rix.
\end{array}
\]
where $S=\diag(I_p,-I_q)$.

Omitting again the tildes, we write
\[
(E,A)\equiv\left(
\mat{ccc}S&0&E_{13}\\0&0&E_{23}\\E_{13}^T&E_{23}^T&E_{33}\rix,
\mat{ccc}0&0&A_{13}\\0&0&A_{23}\\0&0&A_{33}\rix\right).
\]
Conservation of skew-adjointness and full rank of the leading block yields
\[
E_{33}=E_{33}^T,\quad 0=-A_{13}-\dot E_{13},\quad 0=-A_{23}-\dot E_{23},\quad A_{33}^T=-A_{33}-\dot E_{33},\quad\rank E_{23}=r.
\]
Hence, there exists a smooth, pointwise nonsingular matrix function~$V$ with
\[
E_{23}V=[\>I_r\>\>0\>]
\]
leading to
\[
\arraycolsep 4pt
\begin{array}{l}
\widetilde E=
\mat{ccc}I_{p+q}&0&0\\0&I_r&0\\0&0&V^T\rix
\mat{ccc}S&0&E_{13}\\0&0&E_{23}\\E_{13}^T&E_{23}^T&E_{33}\rix
\mat{ccc}I_{p+q}&0&0\\0&I_r&0\\0&0&V\rix=
\mat{cccc}S&0&\widetilde E_{13}&\widetilde E_{14}\\0&0&I_r&0\\\widetilde E_{13}^T&I_r&\widetilde E_{33}&\widetilde E_{34}\\\widetilde E_{14}^T&0&\widetilde E_{34}^T&\widetilde E_{44}\rix,\\
\widetilde A=
\mat{ccc}I_{p+q}&0&0\\0&I_r&0\\0&0&V^T\rix
\mat{ccc}0&0&A_{13}\\0&0&A_{23}\\0&0&A_{33}\rix
\mat{ccc}I_{p+q}&0&0\\0&I_r&0\\0&0&V\rix\\
\qquad\qquad{}-
\mat{ccc}I_{p+q}&0&0\\0&I_r&0\\0&0&V^T\rix
\mat{ccc}S&0&E_{13}\\0&0&E_{23}\\E_{13}^T&E_{23}^T&E_{33}\rix
\mat{ccc}0&0&0\\0&0&0\\0&0&\dot V\rix=
\mat{cccc}0&0&\widetilde A_{13}&\widetilde A_{14}\\0&0&\widetilde A_{23}&\widetilde A_{24}\\0&0&\widetilde A_{33}&\widetilde A_{34}\\0&0&\widetilde A_{43}&\widetilde A_{44}\rix.
\end{array}
\]
Omitting again the tildes, we write
\[
(E,A)\equiv\left(
\mat{cccc}S&0&E_{13}&E_{14}\\0&0&I_r&0\\E_{13}^T&I_r&E_{33}&E_{34}\\E_{14}^T&0&E_{34}^T&E_{44}\rix,
\mat{cccc}0&0&A_{13}&A_{14}\\0&0&A_{23}&A_{24}\\0&0&A_{33}&A_{34}\\0&0&A_{43}&A_{44}\rix.
\right).
\]
Conservation of self-adjointness yields
\[
E_{33}^T=E_{33},\quad E_{44}^T=E_{44},\quad 0=-A_{13}-\dot E_{13},\quad 0=-A_{14}-\dot E_{14},\quad 0=A_{23},\quad 0=A_{24},
\]
and
\[
A_{33}^T=-A_{33}-\dot E_{33},\quad A_{43}^T=-A_{34}-\dot E_{34},\quad A_{44}^T=-A_{44}-\dot E_{44}.
\]
Finally, after a congruence transformation with,
\[
\mat{cccc}I_{p+q}&0&-S^{-1}E_{13}&-S^{-1}E_{14}\\0&I_r&0&0\\0&0&I_r&0\\0&0&0&I_a\rix,
\]
we arrive at
\[
\arraycolsep 4pt
\begin{array}{l}
\widetilde E=
\mat{cccc}I_{p+q}&0&0&0\\0&I_r&0&0\\-E_{13}^TS^{-1}&0&I_r&0\\-E_{14}^TS^{-1}&0&0&I_a\rix
\mat{cccc}S&0&E_{13}&E_{14}\\0&0&I_r&0\\E_{13}^T&I_r&E_{33}&E_{34}\\E_{14}^T&0&E_{34}^T&E_{44}\rix
\mat{cccc}I_{p+q}&0&-S^{-1}E_{13}&-S^{-1}E_{14}\\0&I_r&0&0\\0&0&I_r&0\\0&0&0&I_a\rix\\
\qquad\qquad\qquad\qquad{}=
\mat{cccc}S&0&0&0\\0&0&I_r&0\\0&I_r&\widetilde E_{33}&\widetilde E_{34}\\0&0&\widetilde E_{34}^T&\widetilde E_{44}\rix,\\
\widetilde A=
\mat{cccc}I_{p+q}&0&0&0\\0&I_r&0&0\\-E_{13}^TS^{-1}&0&I_r&0\\-E_{14}^TS^{-1}&0&0&I_a\rix
\mat{cccc}0&0&A_{13}&A_{14}\\0&0&A_{23}&A_{24}\\0&0&A_{33}&A_{34}\\0&0&A_{43}&A_{44}\rix
\mat{cccc}I_{p+q}&0&-S^{-1}E_{13}&-S^{-1}E_{14}\\0&I_r&0&0\\0&0&I_r&0\\0&0&0&I_a\rix\\
\qquad\qquad{}-
\mat{cccc}I_{p+q}&0&0&0\\0&I_r&0&0\\-E_{13}^TS^{-1}&0&I_r&0\\-E_{14}^TS^{-1}&0&0&I_a\rix
\mat{cccc}S&0&E_{13}&E_{14}\\0&0&I_r&0\\E_{13}^T&I_r&E_{33}&E_{34}\\E_{14}^T&0&E_{34}^T&E_{44}\rix
\mat{cccc}0&0&-S^{-1}\dot E_{13}&-S^{-1}\dot E_{14}\\0&0&0&0\\0&0&0&0\\0&0&0&0\rix\\
\qquad\qquad\qquad\qquad{}=
\mat{cccc}0&0&0&0\\0&0&0&0\\0&0&\widetilde A_{33}&\widetilde A_{34}\\0&0&\widetilde A_{43}&\widetilde A_{44}\rix.
\end{array}
\]
The corresponding homogeneous DAE $\widetilde E\dot{\widetilde x}=\widetilde A{\widetilde x}$ has the form
\[
\begin{array}{l}
S\dot{\widetilde x}_1=0,\\
\dot{\widetilde x}_3=0,\\
\dot{\widetilde x}_2+\widetilde E_{33}\dot{\widetilde x}_3+\widetilde E_{34}\dot{\widetilde x}_4=\widetilde A_{33}{\widetilde x}_3+\widetilde A_{34}{\widetilde x}_4,\\
\widetilde E_{34}^T\dot{\widetilde x}_3+\widetilde E_{44}\dot{\widetilde x}_4=\widetilde A_{43}{\widetilde x}_3+\widetilde A_{44}{\widetilde x}_4.
\end{array}
\]
If we take~${\widetilde x}_1$ and~${\widetilde x}_3$ as solutions of the first two equations,
the fourth equation must determine~${\widetilde x}_4$, possibly imposing a suitable
initial condition. Finally, the third equation then fixes~${\widetilde x}_2$
using a suitable initial condition. Hence the dimension of the solution space is at least $p+q+2r$.
But the dimension was assumed to be $p+q$ implying that $r=0$.

Omitting the tildes again, skipping the second and third block row and column,
which are of zero dimension, and renumbering the indices gives
\[
(E,A)\equiv\left(
\mat{ccc}I_p&0&0\\0&-I_q&0\\0&0&E_{33}\rix,
\mat{ccc}0&0&0\\0&0&0\\0&0&A_{33}\rix\right)
\]
which is just (\ref{gcfskew}),
where (\ref{gcfskewq}) follows along the same lines as in the general case
and (\ref{gcfskewp}) follows by the conservation of skew-adjointness.
\end{proof}

The presented canonical forms have some direct consequences in the case that $E$ is pointwise positive semidefinite.
\begin{corollary}\label{cor:index}
Let $(E,A)$ be regular with $E,A\in C({\mathbb I},{\mathbb R}^{n,n})$ sufficiently smooth and let $(E,A)$ be skew-adjoint with $E$ pointwise positive semidefinite. Then in the canonical form \eqref{lcfskew2}
$E_{41}=0$ and the flow associated with the dynamical part of the system is orthogonal.
\end{corollary}
\begin{proof}
Since $Q^TEQ$ is positive semidefinite for positive semidefinite~$E$,
it follows that the blocks in positions $(1,3)$, $(3,1)$, $(4,1)$ and $(1,4)$ of (\ref{lcfskew2}) are zero and that~$S$ is pointwise positive definite. Thus $q=0$ and the flow corresponding to (\ref{odeskew}) is orthogonal.
\end{proof}

\begin{remark}\label{rem:index}{\rm
Corollary~\ref{cor:index} immediately implies also an upper bound on the so-called differentiation index of the DAE, see \cite{KunM06}, which is at most two.
This follows directly from \eqref{lcfskew2}, since with $E_{41}=0$ at most one differentiation of the inhomogeneity is needed.
}
\end{remark}

\begin{example}{\rm
Consider the circuit in Example~\ref{ex1}.  Reordering the rows and columns in the order second, third, first, fourth, and fifth, we obtain the skew-adjoint system in canonical form
\[
\mat{cc|c|cc} C_1 & 0 & 0 & 0 & 0 \\
0 & C_2 & 0 & 0 & 0 \\
\hline
0 & 0 & L & 0 & 0\\
\hline
0 & 0 & 0 & 0 & 0 \\
0 & 0 & 0 & 0 & 0 \rix
\mat{c} \dot V_1 \\  \dot  V_2\\ \hline \dot I \\ \hline  \dot I_G \\  \dot I_R \rix=
\mat{cc|c|cc} 0 & 1 & 0 & -1 & 0 \\
-1 & 0 & 1 & 0 & -1\\
\hline
0 & -1 & -R_L & 0 & 0\\
\hline
1 & 0 & 0 & -R_G & 0\\
0 & 1 & 0 & 0 & -R_R \rix
\mat{c}  V_1 \\  V_2\\ \hline I \\ \hline  I_G \\   I_R \rix
+
\mat{c} 0 \\ 0  \\\hline 0 \\ \hline 1 \\ 0 \rix u.
\]

If $u$ is given and the resistive terms are neglected, i.e.\ $R_L,R_G,R_R=0$, then this is skew-adjoint DAE, in which the last two equations can be solved as $V_2=0$ and $V_1=-u$ and from the first two equations we get
$I_G=-C_1 \dot V_1= C_1 \dot u$ and $I_R=-C_2 \dot V_2=0$  and the dynamic equation is
\[
L \dot I =-V_2 =0,
\]
which has the orthogonal flow $I=1$.
}
\end{example}

\begin{example}{\rm
Consider the linear time-varying DAE \eqref{eq:P-instat-op-general} in Example~\ref{ex:ns} with $A_H=0$ and $C=0$. Performing a full rank decomposition
\[
U^T B =\mat{c} B_1 \\ 0 \rix,
\]
with $B_1$ nonsingular (which corresponds to the partitioning of the velocity into the parts that have divergence zero and nonzero), and applying a congruence transformation with
\[
\mat{cc} U & 0 \\ 0 & I \rix
\]
yields a transformed system
\begin{equation*}
\mat{ccc} M_{11} & M_{12} & 0 \\ M_{21} & M_{22} & 0 \\ 0 & 0 & 0 \rix \mat{c} \dot v_1 \\ \dot v_2 \\ \dot p \rix=
\mat{ccc} J_{11}(t) & J_{12}(t) & -B_1 \\ J_{21}(t) & J_{22}(t) & 0 \\ B_1^T & 0 & 0 \rix \mat{c}  v_1 \\  v_2 \\  p \rix+
\mat{c}  f_1(t) \\  f_2(t) \\  0 \rix.
\end{equation*}
The third equation yields $v_1=0$ and the first equation gives
\[
p=  B_1^{-1} (J_{12}(t) v_2 -M_{12} \dot v_2+f_1(t)),
\]
while the underlying dynamics of the system is described by the skew-adjoint DAE
\[
M_{22} \dot v_2 = J_{22}(t) v_2 +f_2(t)
\]
with constant $M_{22}=M_{22}^T>0$ and pointwise skew-symmetric $J_{22}$.
After a change of basis with the positive definite square root $\smash{M_{22}^{1/2}}$  of $M_{22}$ according to $\smash{\tilde v_2= M_{22}^{1/2}v_2}$
and scaling the equation by its inverse $\smash{M_{22}^{-1/2}}$,
we obtain an ODE system
\[
\dot{\tilde v}_2 =\tilde J_{22}(t) \tilde v_2+ \tilde f_2(t),
\]
with pointwise skew-symmetric $\tilde J_{22}$, which has an orthogonal flow.
}
\end{example}

\section{Conclusions}\label{sec}
We have derived local and global canonical forms under congruence transformations for self-adjoint and skew-adjoint systems of linear variable coefficient differential-algebraic equations. The associated flows for the dynamical part of the system are shown to be symplectic or in the generalized orthogonal groups.  The results are illustrated at the hand of examples from electrical network and flow simulation.

\section*{Acknowledgement}
Volker Mehrmann was partially supported by {\it Deutsche Forschungsgemeinschaft} through
the  Excellence Cluster {\sc Math$^+$} in Berlin and Priority Program 1984 `Hybride und multimodale Energiesysteme:
 Systemtheoretische Methoden f\"ur die Transformation und den Betrieb komplexer Netze'.

\bibliographystyle{plain}

\begin{thebibliography}{10}

\bibitem{AltMU21}
R.~Altmann, V.~Mehrmann, and B.~Unger.
\newblock Port-{H}amiltonian formulations of poroelastic network models.
\newblock {\em Math. and Comp. Modelling Dynamical Systems}, 27:429--452, 2021.
\newblock http://arxiv.org/abs/2012.01949.

\bibitem{AltS17}
R.~Altmann and P.~Schulze.
\newblock On the port-{H}amiltonian structure of the {N}avier-{S}tokes
  equations for reactive flows.
\newblock {\em Systems \& Control Letters}, 100:51--55, 2017.

\bibitem{BeaMXZ18}
C.~Beattie, V.~Mehrmann, H.~Xu, and H.~Zwart.
\newblock Port-{H}amiltonian descriptor systems.
\newblock {\em Math. Control, Signals, Sys.}, 30(17):1--27, 2018.

\bibitem{BreF91}
F.~Brezzi and M.~Fortin.
\newblock {\em {M}ixed and {H}ybrid {F}inite {E}lement {M}ethods}.
\newblock Springer-Verlag, New York, NY, 1991.

\bibitem{Cam87a}
S.~L. Campbell.
\newblock A general form for solvable linear time varying singular systems of
  differential equations.
\newblock {\em SIAM J. Math. Anal.}, 18:1101--1115, 1987.

\bibitem{Dol64}
V.~Dole{\v z}al.
\newblock The existence of a continuous basis of a certain subspace of {$E_r$}
  which depends on a parameter.
\newblock {\em Cas. Pro. Pest. Mat.}, 89:466--468, 1964.

\bibitem{DomHLMMT21}
P.~Domschke, B.~Hiller, J.~Lang, V.~Mehrmann, R.~Morandin, and C.~Tischendorf.
\newblock Gas network modeling: An overview.
\newblock {Preprint}, Collaborative Research Center TRR 154, 2021.

\bibitem{Egg19}
H.~Egger.
\newblock Structure preserving approximation of dissipative evolution problems.
\newblock {\em Numer. Math.}, 143(1):85--106, 2019.

\bibitem{EicF98}
E.~Eich-Soellner and C.~F{\"u}hrer.
\newblock {\em Numerical Methods in Multibody Systems}.
\newblock Teubner Verlag, Stuttgart, Germany, 1998.

\bibitem{EmmM13}
E.~{Emmrich} and V.~{Mehrmann}.
\newblock Operator differential-algebraic equations arising in fluid dynamics.
\newblock {\em Comp. Meth. in Appl. Math.}, 13(4):443--470, 2013.

\bibitem{ForGNQ01}
L~Formaggia, J.-F. Gerbeau, F.~Nobile, and A.~Quarteroni.
\newblock On the coupling of 3{D} and 1{D} {N}avier--{S}tokes equations for
  flow problems in compliant vessels.
\newblock {\em Comp. Meth. Appl. Mech. Eng.}, 191(6-7):561--582, 2001.

\bibitem{ForQV10}
L.~Formaggia, A.~Quarteroni, and A.~Veneziani.
\newblock {\em Cardiovascular Mathematics: Modeling and simulation of the
  circulatory system}, volume~1.
\newblock Springer Science \& Business Media, 2010.

\bibitem{GolSBM03}
G.~Golo, A.~J.~{van der} Schaft, P.~C. Breedveld, and B.~M. Maschke.
\newblock {H}amiltonian formulation of bond graphs.
\newblock In R.~Johansson and A.~Rantzer, editors, {\em Nonlinear and Hybrid
  Systems in Automotive Control}, pages 351--372. Springer, Heidelberg, 2003.

\bibitem{HaiLW02}
E.~Hairer, C.~Lubich, and G.~Wanner.
\newblock {\em Geometric Numerical Integration. Structure-Preserving Algorithms
  for Ordinary Differential Equations}.
\newblock Springer-Verlag, Berlin, Germany, 2002.

\bibitem{Hal03}
B.~C. Hall.
\newblock {\em Lie groups, Lie algebras, and representations: an elementary
  introduction}, volume~10.
\newblock Springer, 2003.

\bibitem{HauMMMMRS19}
S.-A. Hauschild, N.~Marheineke, V.~Mehrmann, J.~Mohring, A.~Moses Badlyan,
  M.~Rein, and M.~Schmidt.
\newblock Port-{H}amiltonian modeling of disctrict heating networks.
\newblock {\em DAE Forum, Porogress in Differential-Algebraic Equations II},
  2020.
\newblock 333--355.

\bibitem{HilH06}
M.~Hiller and K.~Hirsch.
\newblock Multibody system dynamics and mechatronics.
\newblock {\em Z.~Angew.~Math.~Mech.}, 86(2):87--109, 2006.

\bibitem{JacZ12}
B.~Jacob and H.~Zwart.
\newblock {\em Linear port-{H}amiltonian systems on infinite-dimensional
  spaces}.
\newblock Operator Theory: Advances and Applications, 223.
  Birkh{\"a}user/Springer Basel AG, Basel CH, 2012.

\bibitem{Kun20}
P.~Kunkel.
\newblock A smooth version of {S}ylvester's law of inertia and its numerical
  realization.
\newblock {\em Electr. Trans. Num. Anal.}, 36:542--560, 2020.

\bibitem{KunM06}
P.~Kunkel and V.~Mehrmann.
\newblock {\em Differential-Algebraic Equations. Analysis and Numerical
  Solution}.
\newblock EMS Publishing House, Z{\"u}rich, Switzerland, 2006.

\bibitem{KunM08}
P.~Kunkel and V.~Mehrmann.
\newblock Optimal control for unstructured nonlinear differential-algebraic
  equations of arbitrary index.
\newblock {\em Math. Control, Signals, Sys.}, 20:227--269, 2008.

\bibitem{KunMS14}
P.~{Kunkel}, V.~{Mehrmann}, and L.~{Scholz}.
\newblock Self-adjoint differential-algebraic equations.
\newblock {\em Math. Control Signals Syst.}, 26:47--76, 2014.

\bibitem{Lay08}
W.~Layton.
\newblock {\em Introduction to the Numerical Analysis of Incompressible Viscous
  Flows}.
\newblock {SIAM} Publications, Philadelphia, PA, 2008.

\bibitem{LeiR94}
B.J. Leimkuhler and S.~Reich.
\newblock Symplectic integration of constrained {H}amiltonian systems.
\newblock {\em Mathematics of Computation}, 63(208):589--605, 1994.

\bibitem{MehM19}
V.~Mehrmann and R.~Morandin.
\newblock Structure-preserving discretization for port-{H}amiltonian descriptor
  systems.
\newblock In {\em 58th IEEE Conference on Decision and Control (CDC), Nice},
  pages 6863--6868. IEEE, 2019.

\bibitem{OrtSMM01}
R.~Ortega, A.~J.~{van der} Schaft, Y.~Mareels, and B.~M. Maschke.
\newblock Putting energy back in control.
\newblock {\em Control Syst. Mag.}, 21:18--33, 2001.

\bibitem{QuaF04}
A.~Quarteroni and L.~Formaggia.
\newblock Mathematical modelling and numerical simulation of the cardiovascular
  system.
\newblock {\em Handbook of numerical analysis}, 12:3--127, 2004.

\bibitem{QuaMV17}
A.~Quarteroni, A.~Manzoni, and C.~Vergara.
\newblock The cardiovascular system: mathematical modelling, numerical
  algorithms and clinical applications.
\newblock {\em Acta Numerica}, 26:365--590, 2017.

\bibitem{QuaQ09}
A.~Quarteroni and S.~Quarteroni.
\newblock {\em Numerical models for differential problems}, volume~2.
\newblock Springer, 2009.

\bibitem{QuaLRR17}
T.~Quarteroni, A.and~Lassila, S.~Rossi, and R.~Ruiz-Baier.
\newblock Integrated heart—coupling multiscale and multiphysics models for
  the simulation of the cardiac function.
\newblock {\em Comp. Meth. Appl. Mech. Eng.}, 314:345--407, 2017.

\bibitem{Ran00}
R.~Rannacher.
\newblock Finite element methods for the incompressible {N}avier-{S}tokes
  equations.
\newblock In P.~Galdi, J.~Heywood, and R.~Rannacher, editors, {\em Fundamental
  Directions in Mathematical Fluid Mechanics}, pages 191--293.
  Birkh\"auser-Verlag, Basel, Switzerland, 2000.

\bibitem{RooST08}
H.-G. Roos, M.~Stynes, and L.~Tobiska.
\newblock {\em Robust Numerical Methods for Singularly Perturbed Differential
  Equations}.
\newblock Springer Science \& Business Media, 2008.

\bibitem{Sch04}
A.~J.~{van der} Schaft.
\newblock Port-{H}amiltonian systems: network modeling and control of nonlinear
  physical systems.
\newblock In {\em Advanced Dynamics and Control of Structures and Machines},
  {CISM} Courses and Lectures, Vol. 444. Springer Verlag, New York, N.Y., 2004.

\bibitem{Sch06}
A.~J.~{van der} Schaft.
\newblock Port-{H}amiltonian systems: an introductory survey.
\newblock In J.~L.~Verona M.~Sanz-Sole and J.~Verdura, editors, {\em Proc. of
  the International Congress of Mathematicians, vol. III, Invited Lectures},
  pages 1339--1365, Madrid, Spain, 2006.

\bibitem{Sch93}
W.~Schiehlen.
\newblock {\em Advanced multibody system dynamics}.
\newblock Kluwer Academic Publishers, Stuttgart, Germany, 1993.

\bibitem{SchK01}
K.~Schlacher and A.~Kugi.
\newblock Automatic control of mechatronic systems.
\newblock {\em Int.~J.~Appl.~Math.~Comput.~Sci.}, 11(1):131--164, 2001.

\bibitem{Sch19}
L.~Scholz.
\newblock Condensed forms for linear port-{H}amiltonian descriptor systems.
\newblock {\em Electr. J. Lin. Alg.}, 35:65--89, 2019.

\bibitem{Tem77}
R.~Temam.
\newblock {\em Navier-Stokes equations. Theory and numerical analysis}.
\newblock North Holland, Amsterdam, The Netherlands, 1977.

\bibitem{Sch13}
A.~J. {van der S}chaft.
\newblock Port-{H}amiltonian differential-algebraic systems.
\newblock In {\em Surveys in Differential-Algebraic Equations. {I}}, pages
  173--226. Springer-Verlag, Heidelberg, 2013.

\bibitem{SchJ14}
A.~J. {van der S}chaft and D.~Jeltsema.
\newblock Port-{H}amiltonian systems theory: {A}n introductory overview.
\newblock {\em Foundations and Trends in Systems and Control}, 1(2-3):173--378,
  2014.

\end{thebibliography}

\end{document}